\documentclass{amsart}
\usepackage{amsmath}
\usepackage{amsfonts}

\newtheorem{theorem}{Theorem}
\theoremstyle{plain}

\newtheorem{corollary}{Corollary}

\newtheorem{definition}{Definition}

\newtheorem{lemma}{Lemma}

\input{tcilatex}

\begin{document}
\title[Inverse Nodal Problem for a CFDO]{Inverse Nodal Problem for a
Conformable Fractional Diffusion Operator}

\begin{abstract}
In this paper, a diffusion operator including conformable fractional
derivatives of order $\alpha $ $\left( 0<\alpha \leq 1\right) $ is
considered. The asymptotics of the eigenvalues, eigenfunctions and nodal
points of the operator are obtained. Furthermore, an effective procedure for
solving the inverse nodal problem is given.
\end{abstract}

\author{YA\c{S}AR \c{C}AKMAK}
\maketitle

\section{\textbf{Introduction}}

The fractional derivative based on 1695 is widely used in applied
mathematics and mathematical analysis. Since then, many researchers have
developed different types of fractional derivative (see [1]-[4]). Unlike
classical Newtonian derivatives, a fractional derivative is given via an
integral form. For example, well-known Riemann-Liouville fractional
derivative is one of them and is defined by%
\begin{equation*}
D_{t}^{\alpha }\left( f\right) \left( t\right) =\frac{1}{\Gamma \left(
n-\alpha \right) }\frac{d^{n}}{dt^{n}}\int\limits_{a}^{t}\frac{f\left(
x\right) }{\left( t-x\right) ^{\alpha -n+1}}dx,
\end{equation*}%
for $\alpha \in \left[ n-1,n\right) .$

In 2014, Khalil et al. introduced the definition of conformable fractional
derivative\ [5]. In 2015, the basic properties and main results of this
derivative was given by Abdeljawad and Atangana et al. ([6], [7]). The
derivative arises in various fields such as quantum mechanics, dynamical
systems, time scale problems, diffusions, conservation of mass, etc. (see
[8]-[11]).

For about a century, inverse spectral theory for the different types of
operators such as Sturm-Liouville, Dirac and diffusion has been
investigated. The first and important result in this theory belongs to
Ambarzumyan (see [12]). After this study, the theory has been developed by
many authors. In recent years, the direct and inverse problems for the
Sturm-Liouville and Dirac operators which include fractional derivative have
been studied (see [13]-[20]). However, in current literature, there are no
results in the inverse spectral theory for a diffusion operator which
include conformable fractional derivative. These problems play an important
role in mathematics and have many applications in natural sciences and
engineering (see [21]-[25]).

The inverse nodal problems consist in recovering operators from given a
dense set of zeros (nodes or nodal points) of eigenfunctions. In 1988,
McLaughlin gave a solution of inverse nodal problem for the Sturm-Liouville
operator (see [26]). Then, many important results for both the diffusion
operators\ and the Sturm-Liouville operators have been studied by several
researchers (see [27]-[41] and references therein).

In the present paper, we consider a diffusion operator with Dirichlet
conditions which include conformable fractional derivatives of order $\alpha 
$ $\left( 0<\alpha \leq 1\right) $ instead of the ordinary derivatives in a
traditional diffusion operator. We reconstruct the potentials of the
diffusion operator from nodes of its eigenfunctions and give an algorithm
for solving the inverse nodal problem.

We note that the analogous results can be obtained also for other types of
boundary conditions.

\section{\textbf{Preliminaries}}

In this section, Firstly, we recall known some concepts of the conformable
fractional calculus. Then, we introduce a conformable fractional diffusion
operator with Dirichlet boundary conditions on $\left[ 0,\pi \right] $.

\begin{definition}
Let $f:[0,\infty )\rightarrow 
\mathbb{R}
$ be a given function. Then, the conformable fractional derivative of $f$ of
order $\alpha $ with respect to $x$ is defined by%
\begin{equation*}
D_{x}^{\alpha }f(x)=\underset{h\rightarrow 0}{\lim }\dfrac{f(x+hx^{1-\alpha
})-f(x)}{h},\text{ }D_{x}^{\alpha }f(0)=\underset{x\rightarrow 0^{+}}{\lim }%
D^{\alpha }f(x),
\end{equation*}%
for all $x>0,$ $\alpha \in (0,1].$ If this limit exist and finite at $x_{0},$
we say $f$ is $\alpha -$differentiable at $x_{0}.$ Note that if $f$ is
differentiable, then, 
\begin{equation*}
D_{x}^{\alpha }f(x)=x^{1-\alpha }f^{\prime }(x).
\end{equation*}
\end{definition}

\begin{definition}
The conformable fractional Integral starting from $0$ of order $\alpha $ is
defined by%
\begin{equation*}
I_{\alpha }f(x)=\int_{0}^{x}f(t)d_{\alpha }t=\int_{0}^{x}t^{\alpha -1}f(t)dt,%
\text{ for all }x>0.
\end{equation*}
\end{definition}

\begin{lemma}
Let $f:[a,\infty )\rightarrow 
\mathbb{R}
$ be any continuous function. Then, for all $x>a$, we have 
\begin{equation*}
D_{x}^{\alpha }I_{\alpha }f(x)=f(x).
\end{equation*}
\end{lemma}

\begin{lemma}
Let $f:(a,b)\rightarrow 
\mathbb{R}
$ be any differentiable function. Then, for all $x>a$, we have 
\begin{equation*}
I_{\alpha }D_{x}^{\alpha }f(x)=f(x)-f(a).
\end{equation*}
\end{lemma}

\begin{definition}
\textbf{(}$\alpha -$\textbf{integration by parts): }Let $f,g:[a,b]%
\rightarrow 
\mathbb{R}
$ be two conformable fractional differentiable functions. Then,%
\begin{equation*}
\int_{a}^{b}f(x)D_{x}^{\alpha }g(x)d_{\alpha }x=\left. f(x)g(x)\right\vert
_{a}^{b}-\int_{a}^{b}g(x)D_{x}^{\alpha }f(x)d_{\alpha }x.
\end{equation*}
\end{definition}

\begin{definition}
The space $C_{\alpha }^{n}[a,b]$ consists of all functions defined on the
interval $[a,b]$ which are continuously $\alpha -$differentiable up to order 
$n.$
\end{definition}

\begin{definition}
Let $1\leq p<\infty ,$ $a>0.$ The space $L_{\alpha }^{p}\left( 0,a\right) $
consists of all functions $f:\left[ 0,a\right] \rightarrow 
\mathbb{R}
$ satisfying the condition 
\begin{equation*}
\left( \int_{0}^{a}\left\vert f(x)\right\vert ^{p}d_{\alpha }x\right)
^{1/p}<\infty .
\end{equation*}
\end{definition}

\begin{lemma}
\lbrack 43], The space $L_{\alpha }^{p}\left( 0,a\right) $ associated with
the norm function 
\begin{equation*}
\left\Vert f\right\Vert _{p,\alpha }:=\left( \int_{0}^{a}\left\vert
f(x)\right\vert ^{p}d_{\alpha }x\right) ^{1/p}
\end{equation*}%
is a Banach space. Moreover if $p=2$ then $L_{\alpha }^{2}\left( 0,a\right) $
associated with the inner product for $f,$ $g\in L_{\alpha }^{2}\left(
0,a\right) $%
\begin{equation*}
\left\langle f,g\right\rangle :=\int_{0}^{a}f(x)\overline{g(x)}d_{\alpha }x
\end{equation*}%
is a Hilbert space.
\end{lemma}

\begin{definition}
\lbrack 43], Let $p\in 
\mathbb{R}
$ be such that $p\geq 1.$ The Sobolev space $W_{\alpha }^{p}\left(
0,a\right) $ consists of all functions on the interval $\left[ 0,a\right] $,
such that $f(x)$ is absolutely continuous and $D_{x}^{\alpha }f(x)\in
L_{\alpha }^{p}\left( 0,a\right) .$
\end{definition}

More detail knowledge about the conformable fractional calculus can be seen
in [5] and [6].

Now, let us consider the boundary value problem $L_{\alpha }=L_{\alpha
}(p(x),q(x))$\ of the form%
\begin{eqnarray}
&&\text{\ }\left. \ell _{\alpha }y:=-D_{x}^{\alpha }D_{x}^{\alpha }y+\left[
2\lambda p(x)+q(x)\right] y=\lambda ^{2}y\text{, \ \ }0<x<\pi \right.
\medskip \\
&&\text{ }\left. U(y):=y(0)=0\right. \medskip \\
&&\text{ }\left. V(y):=y(\pi )=0\right. \medskip
\end{eqnarray}%
where $\lambda $ is the spectral parameter$,$ $p(x),$ $D_{x}^{\alpha }p(x),$ 
$q(x)\in W_{\alpha }^{2}\left( 0,\pi \right) $ are real valued functions, $%
p(x)\neq $const. and 
\begin{equation}
\int_{0}^{\pi }p(x)d_{\alpha }x=0.
\end{equation}

The operator $L_{\alpha }$ is called as conformable fractional diffusion
operator (CFDO).

Let the functions $S\left( x,\lambda \right) $ and $\psi \left( x,\lambda
\right) $ be the solutions of the equation $(1)$ satisfying the initial
conditions%
\begin{equation}
S\left( 0,\lambda \right) =0,\text{ }D_{x}^{\alpha }S\left( 0,\lambda
\right) =1\text{ and }\psi \left( \pi ,\lambda \right) =0,\text{ }%
D_{x}^{\alpha }\psi \left( \pi ,\lambda \right) =1
\end{equation}%
\newline
respectively.

Denote%
\begin{equation}
\Delta \left( \lambda \right) =W_{\alpha }\left[ S\left( x,\lambda \right)
,\psi \left( x,\lambda \right) \right] =S\left( x,\lambda \right)
D_{x}^{\alpha }\psi \left( x,\lambda \right) -\psi \left( x,\lambda \right)
D_{x}^{\alpha }S\left( x,\lambda \right) .
\end{equation}

Where, the function $W_{\alpha }\left[ \varphi \left( x,\lambda \right)
,\psi \left( x,\lambda \right) \right] $ is called the fractional Wronskian
of the functions $S\left( x,\lambda \right) $ and $\psi \left( x,\lambda
\right) .$ It is proven in [17] that $W_{\alpha }$ does not depend on $x$
and putting $x=0$ and $x=\pi $ in (6) it can be written as%
\begin{equation}
\Delta \left( \lambda \right) =V\left( S\right) =-U\left( \psi \right) .
\end{equation}

\begin{definition}
The function $\Delta \left( \lambda \right) $ is called the characteristic
function of the problem $L_{\alpha }.$
\end{definition}

Let us calculate an asymptotic of the eigenvalues of the problem $L_{\alpha
} $. Firstly, we rewritten equation (1) as%
\begin{equation}
D_{x}^{\alpha }D_{x}^{\alpha }y+\frac{D_{x}^{\alpha }p(x)}{\lambda -p(x)}%
D_{x}^{\alpha }y+\left( \lambda -p(x)\right) ^{2}y=\left(
q(x)+p^{2}(x)\right) y+\frac{D_{x}^{\alpha }p(x)}{\lambda -p(x)}%
D_{x}^{\alpha }y.
\end{equation}%
It is easily shown that the system of functions $\left\{ \cos \left( \frac{%
\lambda }{\alpha }x^{\alpha }-Q(x)\right) ,\text{ }\sin \left( \frac{\lambda 
}{\alpha }x^{\alpha }-Q(x)\right) \right\} $ is a fundamental system for the
differential equation 
\begin{equation}
D_{x}^{\alpha }D_{x}^{\alpha }y+\frac{D_{x}^{\alpha }p(x)}{\lambda -p(x)}%
D_{x}^{\alpha }y+\left( \lambda -p(x)\right) ^{2}y=0
\end{equation}%
\newline
where 
\begin{equation}
Q(x):=\int_{0}^{x}p(t)d_{\alpha }t.
\end{equation}%
\newline
By the method of variation of parameters general solution of equation (8) or
(1) is%
\begin{eqnarray}
&&\left. y(x,\lambda )=c_{1}\cos \left( \tfrac{\lambda }{\alpha }x^{\alpha
}-Q(x)\right) +c_{2}\sin \left( \tfrac{\lambda }{\alpha }x^{\alpha
}-Q(x)\right) \right. \medskip  \notag \\
&&+\int_{0}^{x}\tfrac{\sin \left[ \frac{\lambda }{\alpha }\left( x^{\alpha
}-t^{\alpha }\right) -Q(x)+Q(t)\right] }{\lambda -p(t)}\left[ \left(
q(t)+p^{2}(t)\right) y(t,\lambda )+\tfrac{D_{t}^{\alpha }p(t)}{\lambda -p(t)}%
D_{t}^{\alpha }y(t,\lambda )\right] d_{\alpha }t.
\end{eqnarray}%
\newline
Since $S(x,\lambda )$ is the solution of equation (1) satisfying the initial
conditions (5). From (11), we get%
\begin{eqnarray}
&&\left. S(x,\lambda )=\tfrac{1}{\lambda -p(0)}\sin \left( \tfrac{\lambda }{%
\alpha }x^{\alpha }-Q(x)\right) \right. \medskip  \notag \\
&&+\int_{0}^{x}\tfrac{\sin \left[ \frac{\lambda }{\alpha }\left( x^{\alpha
}-t^{\alpha }\right) -Q(x)+Q(t)\right] }{\lambda -p(t)}\left[ \left(
q(t)+p^{2}(t)\right) S(t,\lambda )+\tfrac{D_{t}^{\alpha }p(t)}{\lambda -p(t)}%
D_{t}^{\alpha }S(t,\lambda )\right] d_{\alpha }t
\end{eqnarray}%
and%
\begin{eqnarray}
&&\left. D_{x}^{\alpha }S(x,\lambda )=\left( \lambda -p(x)\right) \left\{ 
\tfrac{1}{\lambda -p(0)}\cos \left( \tfrac{\lambda }{\alpha }x^{\alpha
}-Q(x)\right) \right. \right. \medskip  \notag \\
&&+\left. \int_{0}^{x}\tfrac{\cos \left[ \frac{\lambda }{\alpha }\left(
x^{\alpha }-t^{\alpha }\right) -Q(t)\right] }{\lambda -p(t)}\left[ \left(
q(t)+p^{2}(t)\right) S(t,\lambda )+\tfrac{D_{t}^{\alpha }p(t)}{\lambda -p(t)}%
D_{t}^{\alpha }S(t,\lambda )\right] d_{\alpha }t\right\} .
\end{eqnarray}

\begin{theorem}
For $\left\vert \lambda \right\vert \rightarrow \infty ,$ the following
asymptotic formula is valid:%
\begin{eqnarray}
&&\left. S(x,\lambda )=\tfrac{1}{\lambda }\sin \left( \tfrac{\lambda }{%
\alpha }x^{\alpha }-Q(x)\right) \right. \medskip  \notag \\
&&+\tfrac{1}{2\lambda ^{2}}\left\{ \left( p(x)+p(0)\right) \sin \left( 
\tfrac{\lambda }{\alpha }x^{\alpha }-Q(x)\right) \right. \medskip  \notag \\
&&-\left( \int_{0}^{x}\left( q(t)+p^{2}(t)\right) d_{\alpha }t\right) \cos
\left( \tfrac{\lambda }{\alpha }x^{\alpha }-Q(x)\right) \medskip  \notag \\
&&+\int_{0}^{x}\left( q(t)+p^{2}(t)\right) \cos \left[ \tfrac{\lambda }{%
\alpha }\left( x^{\alpha }-2t^{\alpha }\right) -Q(x)+2Q(t)\right] d_{\alpha
}t\medskip  \notag \\
&&\left. +\int_{0}^{x}D_{t}^{\alpha }p(t)\sin \left[ \tfrac{\lambda }{\alpha 
}\left( x^{\alpha }-2t^{\alpha }\right) -Q(x)+2Q(t)\right] d_{\alpha
}t\right\} \medskip \\
&&+\tfrac{1}{4\lambda ^{3}}\left\{ \left[ 4p^{2}(0)+\tfrac{2\left(
p(x)+p(0)\right) ^{1+\alpha }-2^{2+\alpha }p^{1+\alpha }(0)+\left(
p(x)-p(0)\right) ^{1+\alpha }}{1+\alpha }\right. \right. \medskip  \notag \\
&&\left. -\tfrac{1}{2}\left( \int_{0}^{x}\left( q(t)+p^{2}(t)\right)
d_{\alpha }t\right) ^{2}\right] \sin \left( \tfrac{\lambda }{\alpha }%
x^{\alpha }-Q(x)\right) \medskip  \notag \\
&&\left. -\left( \int_{0}^{x}\left( q(t)+p^{2}(t)\right) \left(
p(x)+p(0)+2p(t)\right) d_{\alpha }t\right) \cos \left( \tfrac{\lambda }{%
\alpha }x^{\alpha }-Q(x)\right) \right\} \medskip  \notag \\
&&+o\left( \tfrac{1}{\lambda ^{3}}\exp \left( \tfrac{\left\vert \tau
\right\vert }{\alpha }x^{\alpha }\right) \right)  \notag
\end{eqnarray}%
uniformly with respect to $x\in \lbrack 0,\pi ],$ where $\tau =\func{Im}%
\lambda .$
\end{theorem}

\begin{proof}
We denote $\medskip $

$S_{0}(x,\lambda )=\tfrac{\sin \left( \tfrac{\lambda }{\alpha }x^{\alpha
}-Q(x)\right) }{\lambda -p(0)}\medskip $\newline
and$\medskip $

$S_{n}(x,\lambda )\medskip $

$=\int\limits_{0}^{x}\tfrac{\sin \left[ \frac{\lambda }{\alpha }\left(
x^{\alpha }-t^{\alpha }\right) -Q(x)+Q(t)\right] }{\lambda -p(t)}\left[
\left( q(t)+p^{2}(t)\right) S_{n-1}(t,\lambda )+\tfrac{D_{t}^{\alpha }p(t)}{%
\lambda -p(t)}D_{t}^{\alpha }S_{n-1}(t,\lambda )\right] d_{\alpha }t,$ for $%
n=1,2,\ldots .\medskip $

Applying successive approximations method to the equations (12) and taking
into account Taylor's expansion formula for the function $\frac{1}{1-u},$ $%
u\rightarrow 0$, we arrive at the estimates (14).
\end{proof}

The eigenvalues of $L_{\alpha }$ coincide with the zeros of its
characteristic function $\Delta (\lambda )=S(\pi ,\lambda ).$ Hence, using
the formulae (4) and (14) one can establish the following asymptotic%
\begin{eqnarray}
&&\left. \Delta (\lambda )=\tfrac{1}{\lambda }\sin \left( \tfrac{\lambda }{%
\alpha }\pi ^{\alpha }\right) \right. \medskip   \notag \\
&&+\tfrac{1}{2\lambda ^{2}}\left\{ \left( p(\pi )+p(0)\right) \sin \left( 
\tfrac{\lambda }{\alpha }\pi ^{\alpha }\right) \right. \medskip   \notag \\
&&-\left( \int_{0}^{\pi }\left( q(t)+p^{2}(t)\right) d_{\alpha }t\right)
\cos \left( \tfrac{\lambda }{\alpha }\pi ^{\alpha }\right) \medskip   \notag
\\
&&+\int_{0}^{\pi }\left( q(t)+p^{2}(t)\right) \cos \left[ \tfrac{\lambda }{%
\alpha }\left( \pi ^{\alpha }-2t^{\alpha }\right) +2Q(t)\right] d_{\alpha
}t\medskip   \notag \\
&&\left. +\int_{0}^{\pi }D_{t}^{\alpha }p(t)\sin \left[ \tfrac{\lambda }{%
\alpha }\left( \pi ^{\alpha }-2t^{\alpha }\right) +2Q(t)\right] d_{\alpha
}t\right\} \medskip  \\
&&+\tfrac{1}{4\lambda ^{3}}\left\{ \left[ 4p^{2}(0)+\tfrac{2\left( p(\pi
)+p(0)\right) ^{1+\alpha }-2^{2+\alpha }p^{1+\alpha }(0)+\left( p(\pi
)-p(0)\right) ^{1+\alpha }}{1+\alpha }\right. \right. \medskip   \notag \\
&&-\tfrac{1}{2}\left. \left( \int_{0}^{\pi }\left( q(t)+p^{2}(t)\right)
d_{\alpha }t\right) ^{2}\right] \sin \left( \tfrac{\lambda }{\alpha }\pi
^{\alpha }\right) \medskip   \notag \\
&&\left. -\left( \int_{0}^{\pi }\left( q(t)+p^{2}(t)\right) \left( p(\pi
)+p(0)+2p(t)\right) d_{\alpha }t\right) \cos \left( \tfrac{\lambda }{\alpha }%
\pi ^{\alpha }\right) \right\} \medskip   \notag \\
&&+o\left( \tfrac{1}{\lambda ^{3}}\exp \left( \tfrac{\left\vert \tau
\right\vert }{\alpha }\pi ^{\alpha }\right) \right) ,\text{ }\left\vert
\lambda \right\vert \rightarrow \infty .  \notag
\end{eqnarray}%
By the standard method using (15) and Rouche's theorem (see [42]) and taking 
$\Delta (\lambda _{n})=0$\ one can prove that eigenvalues $\lambda _{n}$
have the form%
\begin{equation}
\left. \lambda _{n}=\frac{n\alpha }{\pi ^{\alpha -1}}+\frac{a_{1}-A_{n}^{n}}{%
2n\pi }+\frac{\left( p(\pi )+p(0)\right) a_{1}+2a_{2}}{4n^{2}\pi ^{2-\alpha
}\alpha }+o\left( \frac{1}{n^{2}}\right) \right. ,\text{ }\left\vert
n\right\vert \rightarrow \infty ,
\end{equation}%
where $n\in 
\mathbb{Z}
\backslash \left\{ 0\right\} ,$ $x_{n}^{0}=0,$ $x_{n}^{n}=\pi ,$ $\ j\in 
\mathbb{Z}
,$%
\begin{equation*}
a_{1}=\int_{0}^{\pi }\left( q(t)+p^{2}(t)\right) d_{\alpha }t,\text{ }%
a_{2}=\int_{0}^{\pi }\left( q(t)+p^{2}(t)\right) p(t)d_{\alpha }t,
\end{equation*}%
\begin{equation*}
A_{n}^{j}=\int_{0}^{x_{n}^{j}}\left( q(t)+p^{2}(t)\right) \cos \left( \tfrac{%
2nt^{\alpha }}{\pi ^{\alpha -1}}-2Q(t)\right) d_{\alpha
}t-\int_{0}^{x_{n}^{j}}D_{t}^{\alpha }p(t)\sin \left( \tfrac{2nt^{\alpha }}{%
\pi ^{\alpha -1}}-2Q(t)\right) d_{\alpha }t.
\end{equation*}

\begin{corollary}
According to (16) for sufficiently large $\left\vert n\right\vert $ the
eigenvalues $\lambda _{n}$ are real and simple.
\end{corollary}

\section{\textbf{Main Results}}

In this section, under condition (4) we obtain the asymptotics for the nodal
points of $L_{\alpha }$ and prove a constructive procedure for solving the
inverse nodal problem.

\begin{theorem}
For sufficiently large $\left\vert n\right\vert $, the eigenfunction $%
S(x,\lambda _{n})$ has exactly $\left\vert n\right\vert -1$ nodes $x_{n}^{j}$
in $\left( 0,\pi \right) $:

$0<x_{n}^{1}<x_{n}^{2}<...<x_{n}^{n-1}<\pi $ \ for $n>0$ $\medskip $\newline
and

$0<x_{n}^{-1}<x_{n}^{-2}<...<x_{n}^{n+1}<\pi $ \ for $n<0.\medskip $\newline
Moreover, the numbers $x_{n}^{j}$ satisfy the following asymptotic formula:%
\begin{equation}
\begin{array}{l}
\left. \left( x_{n}^{j}\right) ^{\alpha }=\dfrac{j\pi ^{\alpha }}{n}+\dfrac{%
Q(x_{n}^{j})}{n\pi ^{1-\alpha }}\right. \medskip \\ 
+\tfrac{1}{2n^{2}\pi ^{2-2\alpha }\alpha }\left[ \int_{0}^{x_{n}^{j}}\left(
q(t)+p^{2}(t)\right) d_{\alpha }t-\dfrac{a_{1}}{\pi ^{\alpha }}%
(x_{n}^{j})^{\alpha }-\left( A_{n}^{j}-\tfrac{A_{n}^{n}}{\pi ^{\alpha }}%
(x_{n}^{j})^{\alpha }\right) \right] \medskip \\ 
+\tfrac{1}{2n^{3}\pi ^{3-3\alpha }\alpha ^{2}}\left[ \int_{0}^{x_{n}^{j}}%
\left( q(t)+p^{2}(t)\right) p(t)d_{\alpha }t-\left( a_{2}+\tfrac{\left(
p(\pi )+p(0)\right) a_{1}}{2}\right) \tfrac{(x_{n}^{j})^{\alpha }}{\pi
^{\alpha }}\right] \medskip \\ 
+o\left( \frac{1}{n^{3}}\right) ,%
\end{array}%
\end{equation}%
uniformly with respect to j.
\end{theorem}

\begin{proof}
It is obvious that according to (16) for sufficiently large $\left\vert
n\right\vert $ in the domain $\Gamma _{n}=\left\{ \left. \lambda \right\vert 
\text{ }\left\vert \lambda -\frac{n\alpha }{\pi ^{\alpha -1}}\right\vert
\leq 1\right\} $ there is exactly one eigenvalue $\lambda _{n}.$ Taking into
account the real-valuedness of $p(x),$ $q(x)$ we say that is also an
eigenvalue $\overline{\lambda _{n}}\in \Gamma _{n}$, and hence $\lambda _{n}=%
\overline{\lambda _{n}}.$ Therefore, the functions $S(x,\lambda _{n})$ are
real-valued for sufficiently large $\left\vert n\right\vert .$

Substituting (16) in (14) we get%
\begin{eqnarray}
&&\left. \lambda _{n}S(x,\lambda _{n})=\sin \left( \tfrac{nx^{\alpha }}{\pi
^{\alpha -1}}-Q(x)\right) \right. \medskip  \notag \\
&&+\tfrac{1}{2n\pi ^{1-\alpha }\alpha }\left\{ \left[ \left(
a_{1}-A_{n}^{n}\right) \tfrac{x^{\alpha }}{\pi ^{\alpha }}%
-\int_{0}^{x}\left( q(t)+p^{2}(t)\right) d_{\alpha }t\right] \cos \left( 
\tfrac{nx^{\alpha }}{\pi ^{\alpha -1}}-Q(x)\right) \right. \medskip  \notag
\\
&&+\left( p(x)+p(0)\right) \sin \left( \tfrac{nx^{\alpha }}{\pi ^{\alpha -1}}%
-Q(x)\right) \medskip  \notag \\
&&+\int_{0}^{x}\left( q(t)+p^{2}(t)\right) \cos \left( \tfrac{n\left(
x^{\alpha }-2t^{\alpha }\right) }{\pi ^{\alpha -1}}-Q(x)+2Q(t)\right)
d_{\alpha }t\medskip  \notag \\
&&+\left. \int_{0}^{x}D_{t}^{\alpha }p(t)\sin \left( \tfrac{n\left(
x^{\alpha }-2t^{\alpha }\right) }{\pi ^{\alpha -1}}-Q(x)+2Q(t)\right)
d_{\alpha }t\right\} \medskip \\
&&+\tfrac{1}{4n^{2}\pi ^{2-2\alpha }\alpha ^{2}}\left\{ \left[ \left( \left(
p(\pi )+p(0)\right) a_{1}+2a_{2}\right) \tfrac{x^{\alpha }}{\pi ^{\alpha }}%
+\left( p(x)+p(0)\right) a_{1}\tfrac{x^{\alpha }}{\pi ^{\alpha }}\right.
\right. \medskip  \notag \\
&&-\left. \int_{0}^{x}\left( q(t)+p^{2}(t)\right) \left(
p(x)+p(0)+2p(t)\right) d_{\alpha }t\right] \cos \left( \tfrac{nx^{\alpha }}{%
\pi ^{\alpha -1}}-Q(x)\right) \medskip  \notag \\
&&+\left[ 4p^{2}(0)+\tfrac{2\left( p(x)+p(0)\right) ^{1+\alpha }-2^{2+\alpha
}p^{1+\alpha }(0)+\left( p(x)-p(0)\right) ^{1+\alpha }}{1+\alpha }\right.
\medskip  \notag \\
&&+a_{1}\tfrac{x^{\alpha }}{\pi ^{\alpha }}\int_{0}^{x}\left(
q(t)+p^{2}(t)\right) d_{\alpha }t-\left( a_{1}\tfrac{x^{\alpha }}{\pi
^{\alpha }}\right) ^{2}\medskip  \notag \\
&&\left. \left. -\tfrac{1}{2}\left( \int_{0}^{x}\left( q(t)+p^{2}(t)\right)
d_{\alpha }t\right) ^{2}\right] \sin \left( \tfrac{nx^{\alpha }}{\pi
^{\alpha -1}}-Q(x)\right) \right\} +o\left( \tfrac{1}{n^{2}}\right) ,\text{ }%
\left\vert n\right\vert \rightarrow \infty ,  \notag
\end{eqnarray}%
uniformly in $x\in \left[ 0,\pi \right] $. From $S\left( x_{n}^{j},\lambda
_{n}\right) =0$, we get$\medskip $

$\sin \left( \tfrac{n\left( x_{n}^{j}\right) ^{\alpha }}{\pi ^{\alpha -1}}%
-Q(x_{n}^{j})\right) \medskip $

$+\frac{1}{2n\pi ^{1-\alpha }\alpha }\left\{ \left[ \left(
a_{1}-A_{n}^{n}\right) \tfrac{\left( x_{n}^{j}\right) ^{\alpha }}{\pi
^{\alpha }}-\int\limits_{0}^{x_{n}^{j}}\left( q(t)+p^{2}(t)\right) d_{\alpha
}t\right] \cos \left( \tfrac{n\left( x_{n}^{j}\right) ^{\alpha }}{\pi
^{\alpha -1}}-Q(x_{n}^{j})\right) \right. \medskip $

$+\left( p(x_{n}^{j})+p(0)\right) \sin \left( \tfrac{n\left(
x_{n}^{j}\right) ^{\alpha }}{\pi ^{\alpha -1}}-Q(x_{n}^{j})\right) \medskip $

$+\int\limits_{0}^{x_{n}^{j}}\left( q(t)+p^{2}(t)\right) \cos \left( \tfrac{%
n\left( \left( x_{n}^{j}\right) ^{\alpha }-2t^{\alpha }\right) }{\pi
^{\alpha -1}}-Q(x_{n}^{j})+2Q(t)\right) d_{\alpha }t\medskip $

$+\left. \int\limits_{0}^{x_{n}^{j}}D_{t}^{\alpha }p(t)\sin \left( \tfrac{%
n\left( \left( x_{n}^{j}\right) ^{\alpha }-2t^{\alpha }\right) }{\pi
^{\alpha -1}}-Q(x_{n}^{j})+2Q(t)\right) d_{\alpha }t\right\} \medskip $

$+\frac{1}{4n^{2}\pi ^{2-2\alpha }\alpha ^{2}}\left\{ \left[ \left( \left(
p(\pi )+p(0)\right) a_{1}+2a_{2}\right) \tfrac{\left( x_{n}^{j}\right)
^{\alpha }}{\pi ^{\alpha }}+\left( p(x_{n}^{j})+p(0)\right) a_{1}\tfrac{%
\left( x_{n}^{j}\right) ^{\alpha }}{\pi ^{\alpha }}\right. \right. \medskip $

$-\left. \int\limits_{0}^{x_{n}^{j}}\left( q(t)+p^{2}(t)\right) \left(
p(x_{n}^{j})+p(0)+2p(t)\right) d_{\alpha }t\right] \cos \left( \tfrac{%
n\left( x_{n}^{j}\right) ^{\alpha }}{\pi ^{\alpha -1}}-Q(x_{n}^{j})\right)
\medskip $

$+\left[ 4p^{2}(0)+\tfrac{2\left( p(x_{n}^{j})+p(0)\right) ^{1+\alpha
}-2^{2+\alpha }p^{1+\alpha }(0)+\left( p(x_{n}^{j})-p(0)\right) ^{1+\alpha }%
}{1+\alpha }\right. \medskip $

$+a_{1}\tfrac{\left( x_{n}^{j}\right) ^{\alpha }}{\pi ^{\alpha }}%
\int\limits_{0}^{x_{n}^{j}}\left( q(t)+p^{2}(t)\right) d_{\alpha }t-\left(
a_{1}\tfrac{\left( x_{n}^{j}\right) ^{\alpha }}{\pi ^{\alpha }}\right)
^{2}\medskip $

$-\tfrac{1}{2}\left. \left. \left( \int\limits_{0}^{x_{n}^{j}}\left(
q(t)+p^{2}(t)\right) d_{\alpha }t\right) ^{2}\right] \sin \left( \tfrac{%
n\left( x_{n}^{j}\right) ^{\alpha }}{\pi ^{\alpha -1}}-Qx_{n}^{j})\right)
\right\} +o\left( \tfrac{1}{n^{2}}\right) =0,$ $\left\vert n\right\vert
\rightarrow \infty .\medskip $

If last equality is divided by $\cos \left( \tfrac{n\left( x_{n}^{j}\right)
^{\alpha }}{\pi ^{\alpha -1}}-Q(x_{n}^{j})\right) $ and necessary
arrangements are made, we obtain$\medskip $

$\tan \left( \tfrac{n\left( x_{n}^{j}\right) ^{\alpha }}{\pi ^{\alpha -1}}%
-Q(x_{n}^{j})\right) =\medskip $

$\left\{ 1+\frac{1}{2n\pi ^{1-\alpha }\alpha }\left[ p(x_{n}^{j})+p(0)+\int%
\limits_{0}^{x_{n}^{j}}\left( q(t)+p^{2}(t)\right) \sin \left( \tfrac{%
2nt^{\alpha }}{\pi ^{\alpha -1}}-2Q(t)\right) d_{\alpha }t\right. \right.
\medskip $

$\left. +\int\limits_{0}^{x_{n}^{j}}D_{t}^{\alpha }p(t)\cos \left( \tfrac{%
2nt^{\alpha }}{\pi ^{\alpha -1}}-2Q(t)\right) d_{\alpha }t\right] \medskip $

$+\frac{1}{4n^{2}\pi ^{2-2\alpha }\alpha ^{2}}\left[ 4p^{2}(0)+\tfrac{%
2\left( p(x_{n}^{j})+p(0)\right) ^{1+\alpha }-2^{2+\alpha }p^{1+\alpha
}(0)+\left( p(x_{n}^{j})-p(0)\right) ^{1+\alpha }}{1+\alpha }\right.
\medskip $

$\left. \left. +a_{1}\tfrac{\left( x_{n}^{j}\right) ^{\alpha }}{\pi ^{\alpha
}}\int\limits_{0}^{x_{n}^{j}}\left( q(t)+p^{2}(t)\right) d_{\alpha }t-\left(
a_{1}\tfrac{\left( x_{n}^{j}\right) ^{\alpha }}{\pi ^{\alpha }}\right) ^{2}-%
\tfrac{1}{2}\left( \int\limits_{0}^{x_{n}^{j}}\left( q(t)+p^{2}(t)\right)
d_{\alpha }t\right) ^{2}\right] \right\} ^{-1}\times \medskip $

$\times \left\{ \frac{1}{2n\pi ^{1-\alpha }\alpha }\left[ \left(
A_{n}^{n}-a_{1}\right) \tfrac{\left( x_{n}^{j}\right) ^{\alpha }}{\pi
^{\alpha }}-A_{n}^{j}+\int_{0}^{x_{n}^{j}}\left( q(t)+p^{2}(t)\right)
d_{\alpha }t\right] \right. \medskip $

$+\frac{1}{4n^{2}\pi ^{2-2\alpha }\alpha ^{2}}\left[ \int%
\limits_{0}^{x_{n}^{j}}\left( q(t)+p^{2}(t)\right) \left(
p(x_{n}^{j})+p(0)+2p(t)\right) d_{\alpha }t\right. \medskip $

$\left. \left. -\left( \left( p(\pi )+p(0)\right) a_{1}+2a_{2}\right) \tfrac{%
\left( x_{n}^{j}\right) ^{\alpha }}{\pi ^{\alpha }}-\left(
p(x_{n}^{j})+p(0)\right) a_{1}\tfrac{\left( x_{n}^{j}\right) ^{\alpha }}{\pi
^{\alpha }}+o\left( \tfrac{1}{n^{2}}\right) \right] \right\} ,$ $\left\vert
n\right\vert \rightarrow \infty \medskip $\newline
or$\medskip $

$\tan \left( \tfrac{n\left( x_{n}^{j}\right) ^{\alpha }}{\pi ^{\alpha -1}}%
-Q(x_{n}^{j})\right) =\frac{1}{2n\pi ^{1-\alpha }\alpha }\left[ \left(
A_{n}^{n}-a_{1}\right) \tfrac{\left( x_{n}^{j}\right) ^{\alpha }}{\pi
^{\alpha }}-A_{n}^{j}+\int\limits_{0}^{x_{n}^{j}}\left( q(t)+p^{2}(t)\right)
d_{\alpha }t\right] \medskip $

$+\frac{1}{2n^{2}\pi ^{2-2\alpha }\alpha ^{2}}\left[ \int%
\limits_{0}^{x_{n}^{j}}\left( q(t)+p^{2}(t)\right) p(t)d_{\alpha }t-\left(
a_{2}+\frac{\left( p(\pi )+p(0)\right) a_{1}}{2}\right) \tfrac{\left(
x_{n}^{j}\right) ^{\alpha }}{\pi ^{\alpha }}\right] \medskip $

$+o\left( \tfrac{1}{n^{2}}\right) ,$ $\left\vert n\right\vert \rightarrow
\infty .\medskip $

Taking Taylor's expansion formula for the arctangent into account, we get$%
\medskip $

$\tfrac{n\left( x_{n}^{j}\right) ^{\alpha }}{\pi ^{\alpha -1}}%
-Q(x_{n}^{j})=j\pi +\frac{1}{2n\pi ^{1-\alpha }\alpha }\left[ \left(
A_{n}^{n}-a_{1}\right) \tfrac{\left( x_{n}^{j}\right) ^{\alpha }}{\pi
^{\alpha }}-A_{n}^{j}+\int\limits_{0}^{x_{n}^{j}}\left( q(t)+p^{2}(t)\right)
d_{\alpha }t\right] \medskip $

$+\frac{1}{2n^{2}\pi ^{2-2\alpha }\alpha ^{2}}\left[ \int%
\limits_{0}^{x_{n}^{j}}\left( q(t)+p^{2}(t)\right) p(t)d_{\alpha }t-\left(
a_{2}+\frac{\left( p(\pi )+p(0)\right) a_{1}}{2}\right) \tfrac{\left(
x_{n}^{j}\right) ^{\alpha }}{\pi ^{\alpha }}\right] \medskip $

$+o\left( \tfrac{1}{n^{2}}\right) ,$ $\left\vert n\right\vert \rightarrow
\infty .\medskip $

From the last equality, we arrive at (17).
\end{proof}

\begin{corollary}
From (17) it is clear that the set $X$ of all nodal points is dense in the
interval $\left[ 0,\pi \right] .$
\end{corollary}

For each fixed $x\in \left[ 0,\pi \right] $ and $\alpha \in (0,1]$. We can
choose a sequence $\left\{ j_{n}\right\} \subset X$ so that $\underset{%
\left\vert n\right\vert \rightarrow \infty }{\lim }x_{n}^{j_{n}}=x$. Then,
there exist finite limits and corresponding equalities hold :%
\begin{equation}
Q(x)=\pi ^{1-\alpha }\underset{\left\vert n\right\vert \rightarrow \infty }{%
\lim }\left( n\left( x_{n}^{j_{n}}\right) ^{\alpha }-j_{n}\pi ^{\alpha
}\right) ,
\end{equation}%
\begin{equation}
f(x):=2\pi ^{1-\alpha }\alpha \underset{\left\vert n\right\vert \rightarrow
\infty }{\lim }n\left[ \pi ^{1-\alpha }\left( n\left( x_{n}^{j_{n}}\right)
^{\alpha }-j_{n}\pi ^{\alpha }\right) -Q(x_{n}^{j_{n}})\right] ,
\end{equation}%
\begin{eqnarray}
&&\left. g(x):=\pi ^{1-\alpha }\alpha \underset{\left\vert n\right\vert
\rightarrow \infty }{\lim }n\left\{ 2\pi ^{1-\alpha }\alpha \left[ n\pi
^{1-\alpha }\left( n\left( x_{n}^{j_{n}}\right) ^{\alpha }-j_{n}\pi ^{\alpha
}\right) -Q(x_{n}^{j_{n}})\right] \right. \right. \medskip \\
&&\left. -f(x_{n}^{j_{n}})+A_{n}^{j_{n}}-A_{n}^{n}\tfrac{\left(
x_{n}^{j_{n}}\right) ^{\alpha }}{\pi ^{\alpha }}\right\}  \notag
\end{eqnarray}%
and%
\begin{equation}
f(x)=\int_{0}^{x}\left( q(t)+p^{2}(t)\right) d_{\alpha }t-\tfrac{x^{\alpha }%
}{\pi ^{\alpha }}\int_{0}^{\pi }\left( q(t)+p^{2}(t)\right) d_{\alpha }t,
\end{equation}%
\begin{eqnarray}
&&\left. g(x)=\int_{0}^{x}\left( q(t)+p^{2}(t)\right) p(t)d_{\alpha }t-%
\tfrac{x^{\alpha }}{\pi ^{\alpha }}\int_{0}^{\pi }\left(
q(t)+p^{2}(t)\right) p(t)d_{\alpha }t\right. \medskip \\
&&-\tfrac{x^{\alpha }\left( p(\pi )+p(0)\right) }{\pi ^{\alpha }}%
\int_{0}^{\pi }\left( q(t)+p^{2}(t)\right) d_{\alpha }t.  \notag
\end{eqnarray}

Therefore we can prove the following theorem for the solution of the inverse
nodal problem.

\begin{theorem}
Given any dense subset of nodal points $X_{0}\subset X$ uniquely determines
the functions $p(x)$\ and $q(x)$ a.e. on $\left[ 0,\pi \right] $. Moreover,
these functions can be found by the following algorithm.

\textbf{Step-1:} For each fixed $x\in \left[ 0,\pi \right] $ and $\alpha \in
(0,1],$ choose a sequence $\left( x_{n}^{j_{n}}\right) \subset X_{0}$ such
that $\underset{\left\vert n\right\vert \rightarrow \infty }{\lim }%
x_{n}^{j_{n}}=x,$

\textbf{Step-2: }Find the function $Q(x)$ from (19) and calculate 
\begin{equation}
p(x)=D_{x}^{\alpha }Q(x),
\end{equation}

\textbf{Step-3: }Find the function $f(x)$ from (20) and determine 
\begin{equation}
q(x)-\tfrac{1}{\pi ^{\alpha }}\int_{0}^{\pi }q(t)d_{\alpha
}t:=r(x)=D_{x}^{\alpha }f(x)-p^{2}(x)+\tfrac{1}{\pi ^{\alpha }}\int_{0}^{\pi
}p^{2}(t)d_{\alpha }t,
\end{equation}

\textbf{Step-4: }For each fixed $x\in \left[ 0,\pi \right] $ and $\alpha \in
(0,1],$ $\alpha Q(x)-x^{\alpha }\left( p(\pi )+p(0)\right) \neq 0,$ find $%
g(x)$ from (21) and calculate%
\begin{eqnarray}
&&\left. \tfrac{1}{\pi ^{\alpha }}\int_{0}^{\pi }q(t)d_{\alpha }t=\tfrac{%
\alpha }{\alpha Q(x)-x^{\alpha }\left( p(\pi )+p(0)\right) }\left[
g(x)-\int_{0}^{x}\left( r(t)+p^{2}(t)\right) p(t)d_{\alpha }t\right. \right.
\medskip \\
&&\left. +\tfrac{x^{\alpha }}{\pi ^{\alpha }}\int_{0}^{\pi }\left(
r(t)+p^{2}(t)\right) p(t)d_{\alpha }t+\tfrac{x^{\alpha }\left( p(\pi
)+p(0)\right) }{\pi ^{\alpha }}\int_{0}^{\pi }\left( r(t)+p^{2}(t)\right)
d_{\alpha }t\right] ,  \notag
\end{eqnarray}

\textbf{Step-5: }Calculate the function $q(x)$ via the formula 
\begin{equation}
q(x)=r(x)+\tfrac{1}{\pi ^{\alpha }}\int_{0}^{\pi }q(t)d_{\alpha }t.
\end{equation}

\begin{proof}
Formula (24) it is obvious from (10).

Differentiating (22) we get $\medskip $

$D_{x}^{\alpha }f(x)=q(x)+p^{2}(x)-\tfrac{1}{\pi ^{\alpha }}%
\int\limits_{0}^{\pi }\left( q(t)+p^{2}(t)\right) d_{\alpha }t.\medskip $

Denote $r(x):=q(x)-\tfrac{1}{\pi ^{\alpha }}\int\limits_{0}^{\pi
}q(t)d_{\alpha }t.$ We obtain immediately formula (25).$\medskip $

Substituting the function $q(x)=r(x)-\tfrac{1}{\pi ^{\alpha }}%
\int\limits_{0}^{\pi }q(t)d_{\alpha }t$ in (23) and taking (4) into account
we get formula (26). $\medskip $

Finally, from (25) and (26) we arrive at (27).$\medskip $
\end{proof}
\end{theorem}

\end{document}